\documentclass{birkjour}
 \newtheorem{thm}{Theorem}[section]
 \newtheorem{cor}[thm]{Corollary}
 \newtheorem{lem}[thm]{Lemma}
 
 \theoremstyle{definition}
 \newtheorem{defn}[thm]{Definition}
 \theoremstyle{remark}

 \numberwithin{equation}{section}

\usepackage{hyperref}
\usepackage{amsfonts}
\usepackage{amssymb}
\usepackage{amsmath}

\newcommand{\ket}[1]{ | #1  \rangle}
\newcommand{\bra}[1]{ \langle #1  |}
\newcommand{\braket}[2]{ \left\langle #1  | #2  \right\rangle}
\newcommand{\ii}{\mathbf{i_1}}
\newcommand{\iii}{\mathbf{i_2}}
\newcommand{\jj}{\mathbf{j}}
\newcommand{\ee}{\mathbf{e_1}}
\newcommand{\eee}{\mathbf{e_2}}
\newcommand{\e}[1]{\mathbf{e_{#1}}}
\newcommand{\ik}[1]{\mathbf{i_{#1}}}
\newcommand{\T}{\mathbb{T}}
\newcommand{\C}{\mathbb{C}}
\newcommand{\D}{\mathbb{D}}
\newcommand{\R}{\mathbb{R}}
\newcommand{\nc}{\mathcal{NC}}
\renewcommand{\(}{\left(}
\renewcommand{\)}{\right)}
\newcommand{\oa}{\left\{}
\newcommand{\fa}{\right\}}
\renewcommand{\[}{\left[}
\renewcommand{\]}{\right]}
\renewcommand{\P}[2]{P_{#1}(#2)}
\newcommand{\dett}[1]{\det{\left( #1 \right) }}
\renewcommand{\exp}[1]{\textrm{exp}\oa #1 \fa}
\renewcommand{\sc}[2]{\left\langle#1,#2\right\rangle}

\newcommand{\PP}[2]{P_{#1}\( #2 \)}

\newcommand{\bo} {\ensuremath{{\bf i_1}}}

\newcommand{\eo} {\ensuremath{{\bf e_1}}}
\newcommand{\et} {\ensuremath{{\bf e_2}}}
\newcommand{\iic}{{\rm \bf{i}}_{\bf{1}}^2}
\newcommand{\iiic}{{\rm \bf{i}}_{\bf{2}}^2}

\newcommand{\h}[1]{{\widehat{#1}}}
\newcommand{\hh}{{\widehat{1}}}
\newcommand{\hhh}{{\widehat{2}}}

\newcommand{\mC}{\ensuremath{\mathbb{C}}}

\newcommand{\scalarmath}[2]{\( #1, #2 \)}

\begin{document}

%
%
%
%
%
%
%
%
%

\title[Finite-Dimensional Bicomplex Hilbert Spaces]
{Finite-Dimensional Bicomplex Hilbert\\
Spaces}

\author{Rapha\"el Gervais Lavoie}
\address{D\'epartement de physique \br
Universit\'e du Qu\'ebec \`a Trois-Rivi\`eres \br
C.P. 500 Trois-Rivi\`eres \br
Qu\'ebec G9A 5H7 \br
Canada}
\email{raphael.gervaislavoie@uqtr.ca}

\thanks{DR is grateful to the Natural Sciences and
Engineering Research Council of Canada for financial
support.  RGL would like to thank the Qu\'{e}bec
FQRNT Fund for the award of a postgraduate
scholarship.}

\author{Louis Marchildon}
\address{D\'epartement de physique \br
Universit\'e du Qu\'ebec \`a Trois-Rivi\`eres \br
C.P. 500 Trois-Rivi\`eres \br
Qu\'ebec G9A 5H7 \br
Canada}
\email{louis.marchildon@uqtr.ca}

\author{Dominic Rochon}
\address{D\'epartement de math\'ematiques et d'informatique \br
Universit\'e du Qu\'ebec \`a Trois-Rivi\`eres \br
C.P. 500 Trois-Rivi\`eres \br
Qu\'ebec G9A 5H7 \br
Canada}
\email{dominic.rochon@uqtr.ca}



\date{January 1, 2004}

\begin{abstract}
This paper is a detailed study of finite-dimensional
modules defined on bicomplex numbers.  A number of results
are proved on bicomplex square matrices, linear operators,
orthogonal bases, self-adjoint operators and Hilbert spaces,
including the spectral decomposition theorem.  Applications
to concepts relevant to quantum mechanics, like
the evolution operator, are pointed out.
\end{abstract}

\maketitle

\section{Introduction}

Bicomplex numbers~\cite{Price}, just like quaternions,
are a generalization of complex numbers by means of
entities specified by four real numbers.  These two
number systems, however, are different in two
important ways: quaternions, which form a division
algebra, are noncommutative, whereas bicomplex numbers
are commutative but do not form a division algebra.

Division algebras do not have zero divisors, that is,
nonzero elements whose product is zero.  Many believe
that any attempt to generalize quantum mechanics
to number systems other than complex numbers should
retain the division algebra property.  Indeed
considerable work has been done over the years on
quaternionic quantum mechanics~\cite{Adler}.

In the past few years, however, it was pointed out that several
features of quantum mechanics can be generalized to bicomplex
numbers.  A generalization of Schr\"{o}dinger's equation
for a particle in one dimension was proposed~\cite{Rochon2},
and self-adjoint operators were defined on finite-dimensional
bicomplex Hilbert spaces~\cite{Rochon3}.
Eigenvalues and eigenfunctions of the bicomplex analogue
of the quantum harmonic oscillator Hamiltonian were obtained
in full generality~\cite{GMR}.

The perspective of generalizing quantum mechanics to bicomplex
numbers motivates us in developing further mathematical tools
related to finite-dimensional bicomplex Hilbert spaces and
operators acting on them.  After a brief review of bicomplex
numbers and modules in Section~2, we devote Section~3 to a
number of results in linear algebra that do not depend on the
introduction of a scalar product.  Basic properties of
bicomplex square matrices are obtained and theorems are
proved on bases, idempotent projections and the representation
of linear operators.  In Section~4 we define the bicomplex scalar
product and derive a number of results on Hilbert spaces,
orthogonalization and self-adjoint
operators, including the spectral decomposition theorem.
Section~5 is devoted to applications to unitary operators,
functions of operators and the quantum evolution operator.
We conclude in Section~6.

\section{Basic Notions}\label{Preliminaries}

This section summarizes known properties of bicomplex
numbers and modules, on which the bulk of this paper is
based.  Proofs and additional results can be found
in~\cite{Price, Rochon2, Rochon3, Rochon1}.

\subsection{Bicomplex Numbers}\label{Bicomplex Numbers}

\subsubsection{Definition}\label{Definition of bicomplex numbers}
The set $\T$ of \emph{bicomplex numbers} is defined as
\begin{align}
\mathbb{T}:=\{ w=z_1+z_2\mathbf{i_2}~|~z_1,z_2\in\mathbb{C}(\mathbf{i_1}) \},\label{2.1}
\end{align}
where $\ii$, $\iii$ and $\jj$ are imaginary and hyperbolic
units such that $\iic=-1=\iiic$ and $\mathbf{j}^2=1$.
The product of units is commutative and is defined as
\begin{equation}
\ii\iii=\jj, \qquad \ii\jj=-\iii,
\qquad \iii\jj=-\ii. \label{2.2}
\end{equation}
With the addition and multiplication of two
bicomplex numbers defined in the obvious way,
the set $\mathbb{T}$ makes up a commutative ring.

Three important subsets of $\mathbb{T}$ can be
specified as
\begin{align}
\mathbb{C}(\ik{k}) &:= \{ x+y\ik{k}~|~x,y\in\mathbb{R} \},
\qquad k=1,2 ;\label{2.3}\\
\mathbb{D} &:= \{ x+y\jj~|~x,y\in\mathbb{R} \} .\label{2.4}
\end{align}
Each of the sets $\mathbb{C}(\ik{k})$ is isomorphic
to the field of complex numbers, while $\mathbb{D}$ is
the set of so-called \emph{hyperbolic numbers}.

\subsubsection{Conjugation and Moduli}\label{Bicomplex conjugation}
Three kinds of conjugation can be defined on
bicomplex numbers. With $w$ specified as in~\eqref{2.1}
and the bar ($\,\bar{\mbox{}}\,$) denoting complex
conjugation in $\mathbb{C}(\mathbf{i_1})$,
we define
\begin{equation}
w^{\dag_1}:=\bar{z}_1+\bar{z}_2\mathbf{i_2},\label{2.5}
\qquad w^{\dag_2}:=z_1-z_2\mathbf{i_2},
\qquad w^{\dag_3}:=\bar{z}_1-\bar{z}_2\mathbf{i_2} .
\end{equation}
It is easy to check that each conjugation has the following
properties:
\begin{equation}
(s+t)^{\dag_k}=s^{\dag_k}+t^{\dag_k},
\qquad \left(s^{\dag_k} \right)^{\dag_k}=s,
\qquad (s\cdot t)^{\dag_k}=s^{\dag_k}\cdot t^{\dag_k} .
\label{2.6}
\end{equation}
Here $s,t\in\mathbb{T}$ and $k=1,2,3$.

With each kind of conjugation, one can define a specific
bicomplex modulus as
\begin{subequations}
\begin{align}
|w|_\ii^2&:=w\cdot w^{\dag_2}=z_1^2+z_2^2~\in\C(\ii),\label{2.7a}\\
|w|_\iii^2&:=w\cdot w^{\dag_1}=\left(|z_1|^2-|z_2|^2\right)
+ 2 \, \textrm{Re}(z_1\bar{z}_2)\iii~\in\C(\iii),\label{2.7b}\\
|w|_\jj^2&:=w\cdot w^{\dag_3}=\left(|z_1|^2+|z_2|^2\right)
- 2 \, \textrm{Im}(z_1\bar{z}_2)\jj~\in\D.\label{2.7c}
\end{align}
\end{subequations}
It can be shown that $|s\cdot t|_k^2=|s|_k^2\cdot|t|_k^2$,
where $k=\ii,\iii$ or $\jj$.

In this paper we will often use the Euclidean $\R^4$ norm
defined as
\begin{equation}
|w|:=\sqrt{|z_1|^2+|z_2|^2}=\sqrt{\textrm{Re}(|w|_\jj^2)} \; .
\label{2.8}
\end{equation}
Clearly, this norm maps $\T$ into $\R$.  We have $|w|\geq0$,
and $|w|=0$ if and only if $w=0$. Moreover~\cite{Price},
for all $s,t\in\T$,
\begin{equation}
|s+t|\leq|s|+|t|, \qquad |s\cdot t|\leq \sqrt{2}|s|\cdot|t|.
\label{2.9}
\end{equation}

\subsubsection{Idempotent Basis}\label{Idempotant basis}

Bicomplex algebra is considerably simplified by
the introduction of two bicomplex numbers $\ee$
and $\eee$ defined as
\begin{equation}
\ee:=\frac{1+\jj}{2},\qquad\eee:=\frac{1-\jj}{2}.\label{2.10}
\end{equation}
In fact $\ee$ and $\eee$ are hyperbolic numbers.
They make up the so-called \emph{idempotent basis} of the bicomplex numbers. One easily checks that ($k=1,2$)
\begin{equation}
\mathbf{e}_{\mathbf{1}}^2=\ee,
\quad \mathbf{e}_{\mathbf{2}}^2=\eee,
\quad \ee+\eee=1,
\quad \mathbf{e}_{\mathbf{k}}^{\dag_3}=\e{k} ,
\quad \ee\eee=0 . \label{2.11}
\end{equation}

Any bicomplex number $w$ can be written uniquely as
\begin{equation}
w = z_1+z_2\iii = z_\hh \ee + z_\hhh \eee , \label{2.12}
\end{equation}
where
\begin{equation}
z_\hh= z_1-z_2\ii \quad \mbox{and}
\quad z_\hhh= z_1+z_2\ii \label{2.12a}
\end{equation}
both belong to $\mathbb{C}(\ii)$.  The caret notation
($\hh$ and $\hhh$) will be used systematically in
connection with idempotent decompositions, with the
purpose of easily distinguishing different types
of indices.  As a consequence of~\eqref{2.11}
and~\eqref{2.12}, one can check that if
$\sqrt[n]{z_\hh}$ is an $n$th root of $z_\hh$
and $\sqrt[n]{z_\hhh}$ is an $n$th root of $z_\hhh$,
then $\sqrt[n]{z_\hh} \, \ee + \sqrt[n]{z_\hhh} \, \eee$
is an $n$th root of $w$.

The uniqueness of the idempotent decomposition
allows the introduction of two projection operators as
\begin{align}
P_1: w \in\T&\mapsto z_\hh \in\C(\ii),\label{2.14}\\
P_2: w \in\T&\mapsto z_\hhh \in\C(\ii).\label{2.15}
\end{align}
The $P_k$ ($k = 1, 2$) satisfy
\begin{equation}
[P_k]^2=P_k, \qquad P_1\ee+P_2\eee=\mathbf{Id}, \label{2.16}
\end{equation}
and, for $s,t\in\T$,
\begin{equation}
P_k(s+t)=P_k(s)+P_k(t),
\qquad P_k(s\cdot t)=P_k(s)\cdot P_k(t) .\label{2.17}
\end{equation}

The product of two bicomplex numbers $w$ and $w'$
can be written in the idempotent basis as
\begin{align}
w \cdot w' = (z_\hh \ee + z_\hhh \eee)
\cdot (z'_\hh \ee + z'_\hhh \eee)
= z_\hh z'_\hh \ee + z_\hhh z'_\hhh \eee .
\label{2.20}
\end{align}
Since 1 is uniquely decomposed as $\ee + \eee$,
we can see that $w \cdot w' = 1$ if and only if
$z_\hh z'_\hh = 1 = z_\hhh z'_\hhh$.  Thus $w$ has an inverse
if and only if $z_\hh \neq 0 \neq z_\hhh$, and the
inverse $w^{-1}$ is then equal to
$(z_\hh)^{-1} \ee + (z_\hhh)^{-1} \eee$.  A nonzero $w$ that
does not have an inverse has the property that
either $z_\hh = 0$ or $z_\hhh = 0$, and such a $w$ is
a divisor of zero.  Zero divisors make up the
so-called \emph{null cone} $\nc$.  That terminology comes
from the fact that when $w$ is written as in~\eqref{2.1},
zero divisors are such that $z_1^2 + z_2^2 = 0$.

Any hyperbolic number can be written in the
idempotent basis as $x_\hh \ee + x_\hhh \eee$, with
$x_\hh$ and $x_\hhh$ in~$\R$.  We define the set~$\D^+$
of positive hyperbolic numbers as
\begin{equation}
\D^+:= \{ x_\hh \ee + x_\hhh \eee ~|~ x_\hh, x_\hhh \geq 0 \}.
\label{2.21}
\end{equation}
Since $w^{\dag_3} = \bar{z}_\hh \ee + \bar{z}_\hhh \eee$,
it is clear that $w \cdot w^{\dag_3} \in \D^+$ for any
$w$ in $\T$.

\subsection{$\T$-Modules and Linear Operators}\label{Module}

Bicomplex numbers make up a commutative ring.
What vector spaces are to fields, modules are to rings.
A module defined over the ring $\T$ of bicomplex
numbers will be called a $\T$-\emph{module}.

\begin{defn}
A \emph{basis} of a $\T$-module is a set of linearly
independent elements that generate the
module.\footnote{The term ``basis'' here should
not be confused with the same word appearing in
``idempotent basis''.  Elements of the former
belong to the module, while elements of the latter
are (bicomplex) numbers.}
\label{Definition2.1b}
\end{defn}

A finite-dimensional \emph{free} $\T$-module is a
$\T$-module with a finite basis.  That is, $M$
is a finite-dimensional free $\T$-module if there exist
$n$ linearly independent elements (denoted $\ket{m_l}$) such that
any element $\ket{\psi}$ of $M$ can be written as
\begin{equation}
\ket{\psi} = \sum_{l=1}^n w_l\ket{m_l} ,\label{2.22}
\end{equation}
with $w_l \in \T$.  We have used Dirac's notation for
elements of $M$ which, following~\cite{Rochon3},
we will call \emph{kets}.

An important subset $V$ of $M$ is
the set of all kets for which all $w_l$ in~\eqref{2.22}
belong to $\C(\ii)$.  In other words, $V$ is the set
of all $\ket{\psi}$ so that
\begin{equation}
\ket{\psi} = \sum_{l=1}^n x_l\ket{m_l} , \qquad
x_l \in \C(\ii) .
\label{2p}
\end{equation}
It was shown in~\cite{Rochon3}
that $V$ is a vector space over the complex numbers,
and that any $\ket{\psi} \in \T$ can be decomposed
uniquely as
\begin{equation}
\ket{\psi} = \ee\ket{\psi}_{\widehat{1}}
+ \eee\ket{\psi}_{\widehat{2}}
= \ee \P{1}{\ket{\psi}} + \eee \P{2}{\ket{\psi}} .
\label{2.23}
\end{equation}
Here $\ket{\psi}_{\widehat{k}} \in V$ and $P_k$
is a projector from $M$ to $V$ ($k = 1, 2$).
One can show that ket projectors
and idempotent-basis projectors (denoted with the
same symbol) satisfy
\begin{equation}
\P{k}{s \ket{\psi} + t \ket{\phi}}
= \P{k}{s} \P{k}{\ket{\psi}}
+ \P{k}{t} \P{k}{\ket{\phi}} .\label{2.24}
\end{equation}

A \emph{bicomplex linear operator} $A$ is a mapping
from $M$ to $M$ such that, for any $s, t \in \T$
and any $\ket{\psi}, \ket{\phi} \in M$,
\begin{equation}
A (s \ket{\psi} + t \ket{\phi})
= s  A \ket{\psi} + t A \ket{\phi} .
\label{2.27}
\end{equation}
A bicomplex linear operator $A$ can always be
written as
\begin{equation}
A = \ee A_\hh + \eee A_\hhh
= \ee P_1(A) + \eee P_2(A) ,
\label{2.24b}
\end{equation}
where $P_k(A)$ ($k = 1, 2$) was defined in~\cite{Rochon3} as
\begin{equation}
\P{k}{A}\ket{\psi}=\P{k}{A\ket{\psi}} \qquad
\forall\ket{\psi}\in M .\label{2.29}
\end{equation}
Clearly one can write
\begin{equation}
A \ket{\psi} = \ee A_\hh\ket{\psi}_\hh
+ \eee A_\hhh \ket{\psi}_\hhh . \label{2.28}
\end{equation}

\begin{defn}
A ket $\ket{\psi}$ belongs to
the null cone if either $\ket{\psi}_\hh = 0$ or
$\ket{\psi}_\hhh = 0$.  A linear operator
$A$ belongs to the null cone if either $A_\hh = 0$
or $A_\hhh = 0$.
\label{D2.1c}
\end{defn}

The following definition adapts to modules the
concepts of eigenvector and eigenvalue, most useful
in vector space theory.

\begin{defn}
Let $A:M\to M$ be a bicomplex linear operator and let
\begin{equation}
A \ket{\psi} = \lambda \ket{\psi},\label{2..30}
\end{equation}
with $\lambda\in \mathbb{T}$ and $\ket{\psi}\in M$ such that
$\ket{\psi}\notin\nc$.  Then $\lambda$ is called an
\emph{eigenvalue} of $A$ and $\ket{\psi}$ is called an
\emph{eigenket} of $A$.
\label{Definition2.1a}
\end{defn}

Just as eigenvectors are normally restricted to nonzero
vectors, we have restricted eigenkets to kets that are not in the
null cone.  One can show that the eigenket equation \eqref{2..30}
is equivalent to the following system of two eigenvector equations
($k = 1, 2$):
\begin{equation}
A_\h{k}\ket{\psi}_\h{k}
=\lambda_\h{k} \ket{\psi}_\h{k} \; .\label{2.30}
\end{equation}
Here $\lambda=\ee\lambda_\hh+\eee\lambda_\hhh$, and
$\ket{\psi}_{\widehat{k}}$ and $A_{\widehat{k}}$ have
been defined before.  For a complete treatment of module
theory, see \cite{Hartley, Bourbaki}.

\section{Bicomplex Linear Algebra}\label{Mathematics}

\subsection{Square Matrices}

A bicomplex $n\times n$ square matrix $A$ is an array
of $n^2$ bicomplex numbers $A_{ij}$.  Since each
$A_{ij}$ can be expressed in the idempotent basis,
it is clear that
\begin{equation}
A=\ee A_\hh+\eee A_\hhh, \label{3.1a}
\end{equation}
where $A_\hh$ and $A_\hhh$ are two complex
$n \times n$ matrices (i.e. with elements in~$\C(\ii))$.

\begin{thm}
Let $A=\ee A_\hh+\eee A_\hhh$ be an $n\times n$ bicomplex matrix.
Then $\dett{A}=\ee\dett{A_\hh}+\eee\dett{A_\hhh}$.
\label{Theo3.1}
\end{thm}
\begin{proof}
We follow the proof given in \cite{Schaum}.
Let $\oa C_i \fa$ be the set of columns of $A$, so that
$A\equiv\( C_1,C_2,\dots,C_n\)$. Let the $i$th column
be such that $C_i=\alpha C'_i+\beta C_i''$, with $\alpha,\beta\in\T$.
Since matrix elements belong to a commutative ring, the
determinant function satisfies
\begin{align*}
& \dett{C_1,C_2,\dots,\alpha C'_i+\beta C''_i,\dots,C_n} \\
& = \alpha\dett{C_1,C_2,\dots,C'_i,\dots,C_n}
+ \beta\dett{C_1,C_2,\dots,C''_i,\dots,C_n}.
\end{align*}
With a bicomplex matrix, we can write
$C_1=\ee C'_1+\eee C''_1$, where columns
$C'_1$ and $C'_2$ have entries in $\C (\ii)$.  Hence
\begin{equation*}
\dett{C_1,\dots,C_n}
=\ee\dett{C'_1,\dots,C_n}+\eee\dett{C''_1,\dots,C_n}.
\end{equation*}
Applying this successively to all columns, we find that
\begin{equation*}
\dett{C_1,\dots,C_n}
=\ee\dett{C'_1,\dots,C'_n}+\eee\dett{C''_1,\dots,C''_n},
\end{equation*}
which is our result.
\end{proof}

From Theorem~\ref{Theo3.1} we immediately see that
$\det{(A)}=0$ if and only if $\det{(A_\hh)}=0=\det{(A_\hhh)}$,
and that $\det{(A)}$ is in the null cone if and only if
$\det{(A_\hh)}=0$ or $\det{(A_\hhh)}=0$.  Moreover,
one can easily prove that for any bicomplex square
matrices $A$ and $B$, $\dett{A^T}=\dett{A}$
(the superscript $T$ denotes the
transpose) and $\dett{AB}=\dett{A}\dett{B}$.

\begin{defn}
A bicomplex square matrix is \emph{singular} if its
determinant is in the null cone.
\label{Definition2.2a}
\end{defn}

\begin{thm}
The inverse $A^{-1}$ of a bicomplex square matrix $A$ exists
if and only if $A$ is not singular, and then $A^{-1}$ is given
by $\ee (A_\hh)^{-1} +\eee (A_\hhh)^{-1}$.
\label{Theo3.2}
\end{thm}
\begin{proof}
If $A$ is not singular, then $\det{(A_\hh)} \neq 0$ and
$\det{(A_\hhh)} \neq 0$, so that $(A_\hh)^{-1}$ and
$(A_\hhh)^{-1}$ both exist.  But then
\begin{equation*}
\( \ee (A_\hh)^{-1} + \eee (A_\hhh)^{-1} \)
\( \ee A_\hh+\eee A_\hhh \) = \ee I+\eee I=I .
\end{equation*}
Conversely, if $A^{-1}$ exists, then $A^{-1}A=I$.
Hence
\begin{equation*}
1=\dett{I}=\dett{A^{-1}A}=\dett{A^{-1}}\dett{A} ,
\end{equation*}
from which we deduce that $\dett{A}$ is not in the
null cone, and therefore that $A$ is not singular.
\end{proof}

Note that although we wrote $A^{-1}$ as a left inverse,
we could have written it just as well as a right inverse,
and both inverses coincide.

\subsection{Free $\T$-Modules and Bases}\label{Kets modules and basis}

Throughout this section, $M$ will denote an
$n$-dimensional free $\T$-module and $\{ \ket{m_l} \}$
will denote a basis of~$M$.  Any element $\ket{\psi}$
of $M$ can be expressed as in~\eqref{2.22}.

In a vector space, any nonzero vector can be part
of a basis.  Not so for $\T$-modules.

\begin{thm}
No basis element of a free $\T$-module can belong
to the null cone.
\label{Theo3.6}
\end{thm}
\begin{proof}
Let $\ket{s_p}$ be an element of a basis of $M$
(not necessarily the $\{ \ket{m_l} \}$ basis).
By~\eqref{2.23} we can write
\begin{equation}
\ket{s_p} = \ee\ket{s_p}_{\widehat{1}}
+ \eee\ket{s_p}_{\widehat{2}} .
\end{equation}
Suppose that $\ket{s_p}$ belongs to the null
cone.  Then either $\ket{s_p}_{\widehat{1}} = 0$ or
$\ket{s_p}_{\widehat{2}} = 0$.  In the first case
$\ee \ket{s_p} = 0$ and in the second case
$\eee \ket{s_p} = 0$.  Both these equations contradict
linear independence.
\end{proof}

We now define two important subsets of $M$.

\begin{defn}
For $k = 1, 2$, $V_k := \left\{{\bf e_k}\sum_{l=1}^{n}{x_{l}\ket{m_l}}
\mid x_{l}\in\mC(\bo)\right\}$.  Or succinctly, $V_k := {\bf e_k} V$.
\label{Definition3.1}
\end{defn}

Clearly, $V_k$ is an $n$-dimensional
vector space over $\mC(\bo)$, isomorphic
to $V$ and with ${\bf e_k} \ket{m_l}$ as basis elements.
All three vector spaces $V$, $V_1$ and $V_2$ are
useful.  Many results proved in~\cite{Rochon3} used~$V$
in a crucial way, while the computation of harmonic
oscillator eigenvalues and eigenkets in~\cite{GMR}
was based on infinite-dimensional analogues
of~$V_1$ and~$V_2$.

Although $V$ depends on the choice of basis $\{ \ket{m_l} \}$,
$V_1$ and $V_2$ do not.  This comes from the fact that any
element of $V_1$ (for instance) can be written as
$\ee \ket{\psi}$, with $\ket{\psi}$ in~$M$.  Clearly,
this makes no reference to any specific basis.

The module $M$, defined over the ring $\T$, has
$n$ dimensions.  We now show that the set of elements
of~$M$ can also be viewed as a $2n$-dimensional
vector space over~$\C(\ii)$, which we shall call $M'$.
To see this, we write in the idempotent basis the
coefficients $w_l$ of a general element of~$M$.
Making use of~\eqref{2.12} and~\eqref{2.22}, we get
\begin{equation}
\ket{\psi} = \sum_{l=1}^n (\ee w_{l\hh} + \eee w_{l\hhh} )\ket{m_l}
= \sum_{l=1}^n w_{l\hh} \ee \ket{m_l}
+ \sum_{l=1}^n w_{l\hhh} \eee \ket{m_l} .
\label{2.22b}
\end{equation}
It is not difficult to show that the $2n$
elements $\ee \ket{m_l}$ and $\eee \ket{m_l}$
($l = 1 \ldots n$) are linearly independent over~$\C(\ii)$.
This proves our claim and, moreover, proves

\begin{thm}
$M' = V_1 \oplus V_2$.
\label{Theo3.3}
\end{thm}

It is well known that all bases of a
finite-dimensional vector space have the same number
of elements.  This, however, is not true in general
for modules~\cite{Hartley}.  But for $\T$-modules
we have

\begin{thm}
Let $M$ be a finite-dimensional free $\mathbb{T}$-module.
Then all bases of $M$ have the same number of elements.
\label{Theo dimension}
\label{Theo3.4}
\end{thm}
\begin{proof}
Let $\{ \ket{m_l}, l = 1\ldots n\}$ and $\{ \ket{s_p}, p = 1\ldots m \}$
be two bases of $M$.  We can write
\begin{align*}
M &= \left\{\sum_{p=1}^{m}{w_{p}\ket{s_p}}
\mid w_{p} \in \mathbb{T}\right\} \\
&= \left\{\sum_{p=1}^{m}{(P_1(w_{p})\eo
+ P_2(w_{p})\et)\ket{s_p}} \mid w_{p} \in \mathbb{T}\right\} ,
\end{align*}
where, as usual, $P_1$ and $P_2$ are defined
with respect to the $\ket{m_l}$.  Since
\begin{equation*}
\(P_1(w_{p})\eo+P_2(w_{p})\et\)\ket{s_p}
= P_1(w_{p})\eo\ket{s_p}+P_2(w_{p})\et\ket{s_p}
\end{equation*}
and $P_k(w_{p})\in\mC(\bo)$ for $k=1,2$, we see that
$\{{\bf e_k}\ket{s_p} ~|~ p = 1\ldots m \}$ is a basis of~$V_k$.
But
\begin{equation*}
\mbox{dim} (V_1) = \mbox{dim} (V_2) = \mbox{dim} (V) = n ,
\end{equation*}
whence $m = n$.
\end{proof}

With the projections $P_k$ defined with respect
to the $\ket{m_l}$, it is obvious that
$P_k (\ket{m_l}) = \ket{m_l}$ ($k = 1, 2$).
This is a direct consequence of the identity
$\ket{m_l} = \ee \ket{m_l} + \eee \ket{m_l}$.
Hence $\{ P_k (\ket{m_l}) ~|~ l = 1\ldots n \}$
is a basis of $V$.  It turns out that the
projection of any basis of $M$ is a basis of~$V$.

\begin{thm}
Let $P_1$ and $P_2$ be the projections defined
with respect to a basis $\{ \ket{m_l} \}$ of $M$,
and let $V$ be the associated vector space.  If
$\{ \ket{s_l} \}$ is another basis of~$M$, then
$\{ P_1(\ket{s_l}) \}$ and $\{ P_2(\ket{s_l}) \}$
are bases of $V$.
\label{Theo projection}
\label{Theo3.7}
\end{thm}
\begin{proof}
We give the proof for $P_1$, the one for $P_2$
being similar.  We first show that the $P_1(\ket{s_l})$
are linearly independent, and then that they generate $V$.

Let $\alpha_l\in\mathbb{C}(\bo)$ for $l=1 \dots n$ and let
\begin{equation*}
\sum_{l=1}^{n}\alpha_l \P{1}{\ket{s_l}}=0 .
\end{equation*}
For $l = 1\ldots n$, define $c_l:=\eo\alpha_l + \et\cdot 0$.
Making use of~\eqref{2.24}, it is easy to see that
$\sum_{l=1}^{n}c_l\ket{s_l}=0$, for
\begin{align*}
\PP{1}{\sum_{l=1}^{n}c_l\ket{s_l}}
&= \sum_{l=1}^{n}\P{1}{c_l}\P{1}{\ket{s_l}}
= \sum_{l=1}^{n}\alpha_l\P{1}{\ket{s_l}}= 0 , \\
\PP{2}{\sum_{l=1}^{n}c_l\ket{s_l}}
&= \sum_{l=1}^{n}\P{2}{c_l}\P{2}{\ket{s_l}}
= \sum_{l=1}^{n}0\cdot \P{2}{\ket{s_l}} = 0.
\end{align*}
The linear independence (in $M$) of $\oa\ket{s_l}\fa$
implies that $\forall l$, $c_l=0$ and therefore
$\alpha_l=0$.

To show that the $P_1(\ket{s_l})$ generate $V$,
let $\ket{\psi_1}\in V$ and consider the ket
\begin{equation*}
\ket{\psi}:=\eo\ket{\psi_1}+\et\cdot 0\in M.
\end{equation*}
Since the (bicomplex) span of $\oa\ket{s_l}\fa$ is $M$,
there exist $d_l\in\mathbb{T}$ such that
\begin{equation*}
\sum_{l=1}^{n}d_l\ket{s_l} = \ket{\psi}.
\end{equation*}
Therefore,
\begin{equation*}
\ket{\psi_1} = \P{1}{\ket{\psi}}
= \PP{1}{\sum_{l=1}^{n}d_l\ket{s_l}}
= \sum_{l=1}^{n}\P{1}{d_l}\P{1}{\ket{s_l}} .
\end{equation*}
Thus, the (complex) span of $\oa\P{1}{\ket{s_l}}\fa$
is the vector space $V$ and $\oa \P{1}{\ket{s_l}}\fa$
is a basis of~$V$.
\end{proof}

\begin{cor}
Let $\ket{\psi}$ be in $M$. If $\ket{\psi}_{\widehat{1}}$
$(\ket{\psi}_{\widehat{2}})$ vanishes, then the projection
$P_1$ $(P_2)$ of all components of $\ket{\psi}$ in any
basis vanishes.
\label{Lemma projection}
\label{Corollary3.3}
\end{cor}
\begin{proof}
Let $\{ \ket{s_l} \}$ be any basis of $M$ and let
$\ket{\psi}_{\widehat{1}} = 0$ (the case with
$\ket{\psi}_{\widehat{2}}$ is similar).  One can write
\begin{equation*}
\ket{\psi} = \sum_{l=1}^n{c_l\ket{s_l}} .
\end{equation*}
Making use of~\eqref{2.23} and~\eqref{2.24}, we get
\begin{equation*}
0 = \ket{\psi}_{\widehat{1}} = \P{1}{\ket{\psi}}
= \sum_{l=1}^n{\P{1}{c_l}P_1(\ket{s_l})} .
\end{equation*}
Since the $P_1(\ket{s_l})$ are linearly independent,
we find that $\forall l$, $\P{1}{c_l} = 0$.
\end{proof}

It is well known that two arbitrary bases of a
finite-dimensional vector space are related by
a nonsingular matrix, where in that context
nonsingular means having nonvanishing
determinant.  Definition~\ref{Definition2.2a}
of a singular bicomplex matrix (as one whose
determinant is in the null cone) leads to the following
analogous theorem.

\begin{thm}
Any two bases of $M$ are related by a nonsingular matrix.
\label{Theo basis}
\label{Theo3.5}
\end{thm}
\begin{proof}
Let $\oa \ket{m_l} \fa$ and $\oa \ket{s_l} \fa$ be two bases of $M$.
From Theorem~\ref{Theo dimension}, we know that both bases have the
same dimension $n$. We can write
\begin{equation*}
\ket{m_l} = \sum_{p=1}^n L_{pl} \ket{s_p} ,
\qquad \ket{s_p} = \sum_{j=1}^n N_{jp} \ket{m_j} ,
\label{T;3.2}
\end{equation*}
where $L$ and $N$ are both $n \times n$ bicomplex matrices.
But then
\begin{equation*}
\ket{m_l} = \sum_{p=1}^n L_{pl} \sum_{j=1}^n N_{jp} \ket{m_j}
= \sum_{j=1}^n {\oa \sum_{p=1}^n N_{jp} L_{pl} \fa} \ket{m_j} .
\end{equation*}
This means that for any $l$,
\begin{equation*}
\sum_{j=1}^n {\oa \delta_{jl}-(NL)_{jl} \fa \ket{m_j}} =0 .
\end{equation*}
Since the $\ket{m_j}$ are linearly independent, we
get that $\delta_{jl}-(NL)_{jl} = 0$ for all $l$ and $j$,
or $NL=I$.  Hence $L$ and $N$ are inverses of each other
and, by Theorem~\ref{Theo3.2}, nonsingular.
\end{proof}

\subsection{Linear Operators}

In this section we first prove a result on the composition
of two linear operators, and then establish the equivalence
between linear operators and square matrices for
bicomplex numbers.

\begin{thm}
Let $A,B:M\to M$ be two bicomplex linear operators.
Then for $k = 1, 2$,
\begin{enumerate}
\item $P_k(A+B)=P_k(A)+P_k(B)$,
\item $P_k(A\circ B)=P_k(A)\circ P_k(B)$,
\end{enumerate}
where $A\circ B$ denotes the operator that acts on an arbitrary
$\ket{\psi}$ as $\(A\circ B\)\ket{\psi} = A\(B\ket{\psi}\)$.
\label{Theo3.8}
\end{thm}
\begin{proof}
To prove the first part, we let $\ket{\psi}\in M$ and
make use of~\eqref{2.24} and~\eqref{2.29}.  We get
\begin{align*}
\( \P{k}{A+B} \)\ket{\psi}
&= \P{k}{ (A+B)\ket{\psi} } = \P{k}{A\ket{\psi}+B\ket{\psi}} \\
&= \P{k}{A\ket{\psi}}+\P{k}{B\ket{\psi}}
= \P{k}{A}\ket{\psi}+\P{k}{B}\ket{\psi} \\
&= \left( \P{k}{A}+\P{k}{B} \right)\ket{\psi}.
\end{align*}
To prove the second part we use~\eqref{2.24b}
and~\eqref{2.29} to get
\begin{align*}
\( A\circ B \)\ket{\psi} &=A\(B\ket{\psi}\) \\
&= \[ \ee \P{1}{A}+\eee \P{2}{A} \]
\big\{\[ \ee \P{1}{B}+\eee\P{2}{B} \]\ket{\psi}\big\}\\
&= \[ \ee\P{1}{A}+\eee\P{2}{A} \]
\[ \ee\P{1}{B\ket{\psi}}+\eee\P{2}{B\ket{\psi}} \]\\
&= \ee\P{1}{A}\P{1}{B\ket{\psi}}+\eee\P{2}{A}\P{2}{B\ket{\psi}}.
\end{align*}
Applying $P_k$ on both sides, we find that
$\P{k}{\( A\circ B \)\ket{\psi}}=\P{k}{A}\P{k}{B\ket{\psi}}$
or, equivalently, $\P{k}{A\circ B}=\P{k}{A}\circ\P{k}{B}$.
\end{proof}

\begin{thm}
The action of a linear bicomplex operator on $M$ can be
represented by a bicomplex matrix.
\label{Theo3.9}
\end{thm}
\begin{proof}
Let $A:M\to M$ be a bicomplex linear operator
and let $\oa\ket{m_l}\fa$ be a basis of $M$.
Let $\ket{\psi}$ be in $M$ and let $\ket{\psi '}:=A\ket{\psi}$.

Since $\oa\ket{m_l}\fa$ is a basis of $M$, the
ket $A\ket{m_l}$ can be represented as a linear
combination of the $\ket{m_p}$:
\begin{equation*}
A\ket{m_l} = \sum_{p=1}^n A_{pl} \ket{m_p} .
\end{equation*}
Writing $\ket{\psi}$ as in~\eqref{2.22} and making
use of~\eqref{2.27}, we get
\begin{align*}
\ket{\psi '} &= A\ket{\psi}
= \sum_{l=1}^n w_l A \ket{m_l}
= \sum_{l=1}^n w_l \sum_{p=1}^n A_{pl} \ket{m_p} \\
&= \sum_{p=1}^n \oa \sum_{l=1}^n A_{pl} w_l \fa \ket{m_p} .
\end{align*}
Writing $\ket{\psi '} = \sum_{p=1}^n w_p' \ket{m_p}$
and making use of the linear independence of the
$\ket{m_p}$, we obtain
\begin{equation*}
w_p' = \sum_{l=1}^n A_{pl} w_l .
\end{equation*}
The action of $A$ on $\ket{\psi}$ is thus completely
determined by the matrix whose elements are the
bicomplex numbers $A_{pl}$.
\end{proof}

Clearly, the matrix associated with a linear
operator depends on the basis in which kets are
expressed.  Given a specific basis, however,
it is not difficult to show that the matrix associated
with the operator $A\circ B$ is the product of the
matrices associated with $A$ and $B$.

Let two bases $\ket{m_l}$ and $\ket{s_l}$ be related by
$\ket{m_l} = \sum_{p=1}^n L_{pl} \ket{s_p}$.  Let the
linear operator $A$ be represented by the matrix $A_{pl}$
in $\ket{m_l}$ and by the matrix $\tilde{A}_{pl}$ in
$\ket{s_l}$.  Then one can show that
\begin{equation}
A_{ji} = \sum_{p, l = 1}^n (L^{-1})_{jp}
\tilde{A}_{pl} L_{li} .
\label{3.9}
\end{equation}
Finally, it is not difficult to show that if
$A_\hh = 0$, then $A_{pl\hh} = 0$ for all $p$ and $l$,
in every basis.

\section{Bicomplex Hilbert Spaces}
\subsection{Scalar Product}

The bicomplex scalar product was defined in~\cite{Rochon3}
where, as in this paper, the physicists' convention is used
for the order of elements in the product.

\begin{defn}
Let $M$ be a finite-dimensional free $\mathbb{T}$-module.
Suppose that with each pair $\ket{\psi}$ and $\ket{\phi}$ in $M$,
taken in this order, we associate a bicomplex number
$\(\ket{\psi},\ket{\phi}\)$ which, $\forall \ket{\chi} \in M$
and $\forall \alpha\in\mathbb{T}$, satisfies
\begin{enumerate}
\item $(\ket{\psi},\ket{\phi}+\ket{\chi})
= (\ket{\psi},\ket{\phi})+(\ket{\psi},\ket{\chi})$;
\item $(\ket{\psi},\alpha \ket{\phi})
=\alpha (\ket{\psi},\ket{\phi})$;
\item $(\ket{\psi},\ket{\phi})
= (\ket{\phi},\ket{\psi})^{\dagger_{3}}$;
\item $(\ket{\psi},\ket{\psi})=0$
if and only if $\ket{\psi}=0$ .
\end{enumerate}
Then we say that $\(\ket{\psi},\ket{\phi}\)$ is a
\emph{bicomplex scalar product}.
\label{scalar}
\label{Definition4.1}
\end{defn}

Property~3 implies that
$\(\ket{\psi},\ket{\psi}\)\in\mathbb{D}$.
Definition~\ref{scalar} is very general.  In this paper
we shall be a little more restrictive, by requiring
the bicomplex scalar product to be hyperbolic
positive, that is,
\begin{equation}
(\ket{\psi},\ket{\psi})\in\mathbb{D}^{+}, \qquad
\forall\ket{\psi}\in M.
\label{hyperpositive}
\end{equation}
This may be a more natural generalization of the
scalar product on complex vector spaces, where
$(\ket{\psi}, \ket{\psi})$ is never negative.

\begin{defn}
Let $\{ \ket{m_l} \}$ be a basis of $M$ and let $V$ be
the associated vector space.  We say that a scalar product
is $\mC(\bo)$-\emph{closed} under $V$ if,
$\forall \ket{\psi},\ket{\phi}\in V$, we have
$(\ket{\psi},\ket{\phi})\in\mC(\bo)$.
\label{Definition4.1a}
\end{defn}

We note that the property of being $\mC(\bo)$-closed is
basis-dependent.  That is, a scalar product may be
$\mC(\bo)$-closed under~$V$ defined through a basis
$\{ \ket{m_l} \}$, but not under~$V'$ defined
through a basis $\{ \ket{s_l} \}$.  However, one
can show that for $k = 1, 2$, the following projection
of a bicomplex scalar product:
\begin{equation}
(\cdot,\cdot)_{\widehat{k}}
:=P_k((\cdot,\cdot)): M\times M\longrightarrow \mC(\bo)
\end{equation}
is a \textbf{standard scalar product} on $V_k$ as well
as on $V$.

\begin{thm}
Let $\ket{\psi},\ket{\phi}\in M$, then
\begin{equation}
(\ket{\psi},\ket{\phi})
=\eo(\ket{\psi}_{\widehat{1}},\ket{\phi}_{\widehat{1}})_{\widehat{1}}
+\et(\ket{\psi}_{\widehat{2}},
\ket{\phi}_{\widehat{2}})_{\widehat{2}} \, .
\label{decompEq}
\end{equation}
\label{decomp2}
\label{Theo4.1}
\end{thm}
%

\vspace{-2ex}\noindent\emph{Proof.}
Using Theorem~4 of~\cite{Rochon3}, we have
\begin{align*}
\scalarmath{\ket{\psi}}{\ket{\phi}}
&= (\eo \ket{\psi}_{\widehat{1}}+\et \ket{\psi}_{\widehat{2}},
\eo \ket{\phi}_{\widehat{1}}+\et \ket{\phi}_{\widehat{2}}) \\
&= \ee\scalarmath{\ket{\psi}_\hh}{\ket{\phi}_\hh}
+ \eee\scalarmath{\ket{\psi}_\hhh}{\ket{\phi}_\hhh}\\
&= \ee\oa \ee P_1 \! \scalarmath{(\ket{\psi}_\hh}{\ket{\phi}_\hh)}
+ \eee P_2 \! \scalarmath{(\ket{\psi}_\hh}{\ket{\phi}_\hh)} \fa\\
& \qquad + \eee\oa \ee P_1 \! \scalarmath{(\ket{\psi}_\hhh}{\ket{\phi}_\hhh)}
+ \eee P_2 \! \scalarmath{(\ket{\psi}_\hhh}{\ket{\phi}_\hhh)} \fa\\
&= \ee P_1 \! \scalarmath{(\ket{\psi}_\hh}{\ket{\phi}_\hh)}
+ \eee P_2 \! \scalarmath{(\ket{\psi}_\hhh}{\ket{\phi}_\hhh)}\\
&= \ee\scalarmath{\ket{\psi}_\hh}{\ket{\phi}_\hh}_\hh
+ \eee\scalarmath{\ket{\psi}_\hhh}{\ket{\phi}_\hhh}_\hhh .
\tag*{\qed}
\end{align*}

\medskip
Theorem~\ref{Theo4.1} is true whether the bicomplex scalar product
is $\mC(\bo)$-closed under~$V$ or not.  When it is $\mC(\bo)$-closed,
we have for $k = 1, 2$
\begin{equation}
\scalarmath{\ket{\psi}}{\ket{\phi}}_{\widehat{k}}
= \P{k}{\scalarmath{\ket{\psi}}{\ket{\phi}}}
= \scalarmath{\ket{\psi}}{\ket{\phi}} ,
\qquad \forall \ket{\psi},\ket{\phi}\in V .
\label{4.5}
\end{equation}

\begin{cor}
A ket $\ket{\psi}$ is in the null cone if and only if
$(\ket{\psi} ,\ket{\psi})$ is in the null none.
\label{Th4c}
\end{cor}
\begin{proof}
By Theorem~\ref{Theo4.1} we have
\begin{equation}
\scalarmath{\ket{\psi}}{\ket{\psi}}
= \ee\scalarmath{\ket{\psi}_\hh}{\ket{\psi}_\hh}_\hh
+ \eee\scalarmath{\ket{\psi}_\hhh}{\ket{\psi}_\hhh}_\hhh .
\label{Eq4d}
\end{equation}
If $\ket{\psi}$ is in the null cone, then
$\ket{\psi}_\h{k} = 0$ for $k = 1$ or~2. By~\eqref{Eq4d}
then, $\e{k} \scalarmath{\ket{\psi}}{\ket{\psi}} = 0$.

Conversely, if $\scalarmath{\ket{\psi}}{\ket{\psi}}$
is not in the null cone, then by~\eqref{Eq4d}
\begin{equation*}
\scalarmath{\ket{\psi}_\h{k}} {\ket{\psi}_\h{k}}_\h{k} \neq 0,
\qquad k = 1, 2.
\end{equation*}
But then $\scalarmath{\ket{\psi}_\h{k}} {\ket{\psi}_\h{k}} \neq 0$,
and therefore ${\ket{\psi}_\h{k}} \neq 0$ ($k = 1, 2$).
\end{proof}

\subsection{Hilbert Spaces}

\begin{thm}
Let $M$ be a finite-dimensional free $\mathbb{T}$-module,
let $\left\{\ket{m_l}\right\}$ be a basis of $M$ and let
$V$ be the vector space associated with $\left\{\ket{m_l}\right\}$
through~\eqref{2p}.  Then for $k = 1, 2$,
$(V, \scalarmath{\cdot}{\cdot}_{\widehat{k}})$ and
$(V_k, \scalarmath{\cdot}{\cdot}_{\widehat{k}})$
are complex $(\mC(\bo))$ pre-Hilbert spaces.
\label{Pre}
\label{Theo4.2}
\end{thm}
\begin{proof}
Since $(\cdot,\cdot)_{\widehat{k}}$ is a standard scalar product
when the vector space $M'$ of Theorem~\ref{Theo3.3} is restricted
to $V$ or $V_k$, then
$(V, \scalarmath{\cdot}{\cdot}_{\widehat{k}})$ and
$(V_k, \scalarmath{\cdot}{\cdot}_{\widehat{k}})$ are complex
$(\mC(\bo))$ pre-Hilbert spaces.
\end{proof}

\begin{cor}
$(V, \scalarmath{\cdot}{\cdot}_{\widehat{k}})$ and
$(V_k, \scalarmath{\cdot}{\cdot}_{\widehat{k}})$ are
complex $(\mC(\bo))$ Hilbert spaces.
\label{Hilbert}
\label{Corollary4.1}
\end{cor}
\begin{proof}
Theorem~\ref{Pre} implies that pre-Hilbert spaces
$(V, \scalarmath{\cdot}{\cdot}_{\widehat{k}})$
and $(V_k, \scalarmath{\cdot}{\cdot}_{\widehat{k}})$
are finite-dimensional normed spaces over $\mC(\bo)$.
Therefore they are also complete metric spaces~\cite{Herget}.
Hence $V$ and $V_k$ are complex $(\mC(\bo))$ Hilbert spaces.
\end{proof}

Let $\ket{\psi_k}$ and $\ket{\phi_k}$ be in $V_k$ for
$k = 1, 2$.  On the direct sum of the two Hilbert spaces
$V_1$ and $V_2$, one can define a scalar product as follows:
\begin{equation}
( \ket{\psi_1} \oplus \ket{\psi_2},
\ket{\phi_1} \oplus \ket{\phi_2} )
= ( \ket{\psi_1}, \ket{\phi_1})_{\widehat{1}}
+ ( \ket{\psi_2}, \ket{\phi_2})_{\widehat{2}} \, .
\label{4.6}
\end{equation}
Then $M'=V_1\oplus V_2$ is a Hilbert space~\cite{Conway}.

From a set-theoretical point of view, $M$ and $M'$ are
identical.  In this sense we can say, perhaps improperly,
that the \textbf{module} $M$ can be decomposed into the
direct sum of two classical Hilbert spaces, i.e.
$M=V_1\oplus V_2.$  Now let us consider the following
\textbf{norm} on the vector space $M'$:
\begin{equation}
\big{|}\big{|}\ket{\phi}\big{|}\big{|}
:= \frac{1}{\sqrt{2}}\sqrt{\scalarmath{\ket{\phi}_{\widehat{1}}}
{\ket{\phi}_{\widehat{1}}}_{\widehat{1}}
+ \scalarmath{\ket{\phi}_{\widehat{2}}}
{\ket{\phi}_{\widehat{2}}}_{\widehat{2}}} \; .
\end{equation}
Making use of this norm, we can define a metric on $M$:
\begin{equation}
d(\ket{\phi},\ket{\psi})
=\big{|}\big{|}\ket{\phi}-\ket{\psi}\big{|}\big{|}.
\end{equation}
With this metric, $M$ is \textbf{complete} and therefore
a \textbf{bicomplex Hilbert space}.

We note that a bicomplex scalar product is
\textbf{completely characterized} by the two scalar products
$\scalarmath{\cdot}{\cdot}_{\widehat{k}}$ on $V$.
In fact if $\scalarmath{\cdot}{\cdot}_{\widehat{1}}$ and
$\scalarmath{\cdot}{\cdot}_{\widehat{2}}$ are two arbitrary
scalar products on $V$, then
$\scalarmath{\cdot}{\cdot}$ defined in~\eqref{decompEq}
is a bicomplex scalar product on $M$.

As a direct application of this decomposition, we obtain the following important result.

\begin{thm}
Let $f:M\rightarrow \mathbb{T}$ be a linear functional on $M$.
Then there is a unique $\ket{\psi}\in M$ such that $\forall \ket{\phi}$,
$f(\ket{\phi})=\scalarmath{\ket{\psi}}{\ket{\phi}}$.
\label{Theo lin func}
\end{thm}
\begin{proof}
We make use of the analogue theorem on $V$~\cite[p.~215]{Herget},
with the functional $P_k(f)$ restricted to $V$.
The theorem shows that for each $k = 1, 2$,
there is a unique $\ket{\psi_k}\in V$
such that
\begin{equation*}
P_k (f) ({\ket{\phi}}_{\widehat{k}})
= \scalarmath{\ket{\psi_k}}
{{\ket{\phi}}_{\widehat{k}}}_{\widehat{k}} .
\end{equation*}
Making use of Theorem~\ref{decomp2}, we find that
$\ket{\psi}:=\eo\ket{\psi_1} + \et\ket{\psi_2}$
has the desired properties.
\end{proof}

\subsection{Orthogonalization}

Just like in vector spaces, a basis in $M$ can
always be orthogonalized.

\begin{thm}
Let $M$ be a finite-dimensional free $\T$-module and
let $\{ \ket{s_l} \}$ be an arbitrary basis of $M$.
Then one can always find bicomplex linear combinations
of the $\ket{s_l}$ which make up an orthogonal
basis.
\label{Theo3.3.1}
\end{thm}
\begin{proof}
Making use of Theorem~\ref{Theo3.3} and Corollary~\ref{Hilbert},
we see that $M=V_1\oplus V_2$, with $V_k$ a complex
Hilbert space.  By Theorem~\ref{Theo3.7},
$\{ {\bf e_k}\ket{s_l}_{\widehat{k}} \}$
is a basis of $V_k$ ($k=1,2$).  Bases in vector spaces can
always be orthogonalized.  So let
$\{{\bf e_k}\ket{s'_l}_{\widehat{k}} \}$
be an orthogonal basis made up of linear combinations of the
${\bf e_k}\ket{s_l}_{\widehat{k}}$.
For all $l\in\{1\ldots n\}$ and for $p \neq l$, we see that
\begin{align*}
& \scalarmath{{\bf e_1}\ket{s'_l}_{\widehat{1}}
+ {\bf e_2}\ket{s'_l}_{\widehat{2}}\,}
{\,{\bf e_1}\ket{s'_l}_{\widehat{1}}
+ {\bf e_2}\ket{s'_l}_{\widehat{2}}} \\
& \qquad = \scalarmath{\eo \ket{s'_l}_{\widehat{1}}}
{\eo \ket{s'_l}_{\widehat{1}}}
+ \scalarmath{\et \ket{s'_l}_{\widehat{2}}}
{\et \ket{s'_l}_{\widehat{2}}}
\end{align*}
is not in the null cone, and that
\begin{align*}
& \scalarmath{{\bf e_1}\ket{s'_l}_{\widehat{1}}
+ {\bf e_2}\ket{s'_l}_{\widehat{2}}\,}
{\,{\bf e_1}\ket{s'_p}_{\widehat{1}}
+ {\bf e_2}\ket{s'_p}_{\widehat{2}}} \\
& \qquad = \scalarmath{\eo \ket{s'_l}_{\widehat{1}}}
{\eo \ket{s'_p}_{\widehat{1}}}
+ \scalarmath{\et \ket{s'_l}_{\widehat{2}}}
{\et \ket{s'_p}_{\widehat{2}}}
\end{align*}
vanishes.  This shows that the set
$\{ {\bf e_1}\ket{s'_l}_{\widehat{1}}
+ {\bf e_2}\ket{s'_l}_{\widehat{2}}\}$
is an orthogonal basis of $M$.
\end{proof}

It is interesting to see explicitly how the Gram-Schmidt
orthogonalization process can be applied.  Let $\oa \ket{m_l} \fa$
be a basis of $M$. We have shown in Theorem~\ref{Theo3.6} that
no $\ket{m_l}$, and therefore no $(\ket{m_l} , \ket{m_l})$,
can belong to the null cone.  Let $\ket{m'_1}=\ket{m_1}$
and let us define
\begin{equation*}
\ket{m'_2}=\ket{m_2}-\frac{\scalarmath{\ket{m'_1}}{\ket{m_2}}}
{\scalarmath{\ket{m'_1}}{\ket{m'_1}}}\ket{m'_1} .
\end{equation*}
Clearly, $\ket{m'_2}$ exists and
$\scalarmath{\ket{m'_1}}{\ket{m'_2}}=0$.
Moreover, $\scalarmath{\ket{m'_2}}{\ket{m'_2}}$ is not in the null
cone.  If it were, we would have for instance
(by Corollary~\ref{Th4c}) $\eee\ket{m'_2}=0$.
But then
\begin{align*}
0 &= \eee\( \ket{m_2}-\frac{\scalarmath{\ket{m'_1}}{\ket{m_2}}}
{\scalarmath{\ket{m'_1}}{\ket{m'_1}}}\ket{m'_1} \) \\
&= \eee\ket{m_2}+\( -\frac{\scalarmath{\ket{m'_1}}{\ket{m_2}}}
{\scalarmath{\ket{m'_1}}{\ket{m'_1}}}\eee \)\ket{m_1}.
\end{align*}
This is impossible, since $\ket{m_1}$ and $\ket{m_2}$ are linearly
independent.  We can verify that $\ket{m'_1}$ and $\ket{m'_2}$
are also linearly independent. Indeed let $w_1\ket{m'_1}+w_2\ket{m'_2}=0$.
Taking the scalar product of this equation with $\ket{m'_l}$ $(l=1,2)$,
we find that $w_l\scalarmath{\ket{m'_l}}{\ket{m'_l}}=0$.
Because $\ket{m'_l}$ is not in the null-cone, $w_l$ must vanish.

Now we can make an inductive argument to generate an
orthogonal basis. Suppose that we have $k$ linear combination
of $\ket{m_1},\dots,\ket{m_k}$, denoted $\ket{m'_1},\dots,\ket{m'_k}$,
that are mutually orthogonal, linearly independent and not
in the null cone. Let us define
\begin{equation*}
\ket{m'_{k+1}} = \ket{m_{k+1}}
- \frac{\scalarmath{\ket{m'_1}}{\ket{m_{k+1}}}}
{\scalarmath{\ket{m'_1}}{\ket{m'_1}}}\ket{m'_1}
- \dots - \frac{\scalarmath{\ket{m'_k}}{\ket{m_{k+1}}}}
{\scalarmath{\ket{m'_k}}{\ket{m'_k}}}\ket{m'_k}.
\end{equation*}
Clearly, $\ket{m'_{k+1}}$ exists. We now show that
$\ket{m'_1},\dots,\ket{m'_{k+1}}$ are (i) mutually orthogonal,
(ii) not in the null cone and (iii) linearly independent.

To prove (i), it is enough to note that
$\scalarmath{\ket{m'_l}}{\ket{m'_{k+1}}}=0$ for $1\leq l\leq k$.
To prove (ii), let's assume (for instance) that $\eee\ket{m'_{k+1}}=0$.
We then have
\begin{equation*}
0=\eee\ket{m_{k+1}}
- \frac{\eee\scalarmath{\ket{m'_1}}{\ket{m_{k+1}}}}
{\scalarmath{\ket{m'_1}}{\ket{m'_1}}}\ket{m'_1}
- \dots - \frac{\eee\scalarmath{\ket{m'_k}}{\ket{m_{k+1}}}}
{\scalarmath{\ket{m'_k}}{\ket{m'_k}}}\ket{m'_k}.
\end{equation*}
Because the $\ket{m'_l}$ $(l\leq k)$ are linear combinations
of the $\ket{m_l}$, this implies that
\begin{equation*}
0=\eee\ket{m_{k+1}} + \sum_{l=1}^k{w_l\ket{m_l}},
\end{equation*}
for some coefficients $w_l$ (possibly null).
But this equation is impossible because
$\ket{m_{k+1}}$ and $\ket{m_l}$ $(l\leq k)$
are linearly independent.

The proof of (iii), that the $\ket{m'_l}$ $(l\leq k+1)$ are linearly independent, can be carried out just like the one that
$\ket{m'_1}$ and $\ket{m'_2}$ are.  This completes the
orthogonalization process.

Going back to the end of Theorem~\ref{Theo3.3.1}, we can see that
any set like
\begin{equation}
\{ {\bf e_1}\ket{s'_{l_1}}_{\widehat{1}}
+ {\bf e_2}\ket{s'_{l_2}}_{\widehat{2}}\} ,
\label{equation-base2}
\end{equation}
with $l_1$ not always equal to $l_2$,
will give a new orthogonal basis of $M$.
Following this procedure, it is possible to construct
$n!$ different orthogonal bases of $M$.  Of course,
there are an infinite number of bases of~$M$.

The following theorem shows that an orthogonal
basis can always be orthonormalized.

\begin{thm}
Any ket $\ket{\psi}$ not in the null cone can
be normalized.
\label{Theo2.3}
\end{thm}
\begin{proof}
Since $\scalarmath{\ket{\psi}}{\ket{\psi}}\in\D^+$ and
$\ket{\psi}$ is not in the null cone, we can write
\begin{equation}
\scalarmath{\ket{\psi}}{\ket{\psi}} = a \ee + b \eee ,
\end{equation}
with $a>0$ and $b>0$.  It is easy to check that
the ket
\begin{equation*}
\ket{\phi} = \left( \frac{1}{\sqrt{a}} \ee
+ \frac{1}{\sqrt{b}} \eee \right) \ket{\psi}
\end{equation*}
satisfies $(\ket{\phi}, \ket{\phi}) = 1$.
\end{proof}

Note that normalization would be impossible if
the scalar product were outside $\D^+$, that is,
if either $a$ or $b$ were negative.

\subsection{Self-Adjoint Operators}

In Theorem~\ref{Theo lin func} we showed that with
finite-dimensional free $\T$-modules, linear
functionals are in one-to-one correspondence with
kets and act like scalar products.  This allows for
the introduction of Dirac's bra notation and the
alternative writing of the scalar product
$(\ket{\psi}, \ket{\phi})$ as $\braket{\psi}{\phi}$.

In~\cite{Rochon3} the bicomplex \emph{adjoint}
operator $A^*$ of $A$ was introduced as the unique
operator that satisfies
\begin{equation}
(\ket{\psi}, A \ket{\phi})
= (A^* \ket{\psi}, \ket{\phi}),\qquad
\forall\ket{\psi}, \ket{\phi} \in M .
\label{4.13}
\end{equation}
In finite-dimensional free $\T$-modules
the adjoint always exists, is linear and satisfies
\begin{equation}
(A^*)^* = A, \quad
(sA + tB)^* = s^{\dag_3}A^* + t^{\dag_3} B^*,
\quad (AB)^*=B^*A^* .
\label{4.14}
\end{equation}
Moreover,
\begin{equation}
\P{k}{A}^* =\P{k}{A^*},\qquad k=1,2 ,
\label{4.15}
\end{equation}
where $\P{k}{A}^*$ is the $\C(\ii)$ adjoint on $V$.

\begin{lem}
Let $\ket{\psi},\ket{\phi}\in M$. Define an operator
$\ket{\phi}\bra{\psi}$ so that its action on an arbitrary
ket $\ket{\chi}$ is given by
$(\ket{\phi} \bra{\psi}) \ket{\chi}
= \ket{\phi}\(\braket{\psi}{\chi}\)$.
Then $\ket{\phi}\bra{\psi}$ is a linear operator on $M$.
\label{Lemma4.1}
\end{lem}
\begin{proof}
For any $\ket{\chi_1}$ and $\ket{\chi_2}$ in $M$
and for any $\alpha_1$ and $\alpha_2$ in $\T$, we have
\begin{align*}
(\ket{\phi}\bra{\psi}) \( \alpha_1 \ket{\chi_1}
+ \alpha_2 \ket{\chi_2} \)
&= \ket{\phi}\oa \bra{\psi}\( \alpha_1 \ket{\chi_1}
+ \alpha_2 \ket{\chi_2} \) \fa \\
&= \ket{\phi} \oa \alpha_1 \braket{\psi}{\chi_1}
+ \alpha_2 \braket{\psi}{\chi_2} \fa \\
&= \alpha_1 \ket{\phi}\( \braket{\psi}{\chi_1} \)
+ \alpha_2 \ket{\phi}\( \braket{\psi}{\chi_2} \) \\
&= \alpha_1 \(\ket{\phi}\bra{\psi}\) \ket{\chi_1}
+ \alpha_2 \(\ket{\phi}\bra{\psi}\) \ket{\chi_2} .
\end{align*}
Ring commutativity allowed us to move scalars freely
around kets.
\end{proof}

\begin{thm}
Let $\oa\ket{u_l}\fa$ be an orthonormal basis of $M$.  Then
\begin{equation*}
\sum_{l=1}^n{\ket{u_l}\bra{u_l}} = I .
\end{equation*}
\label{Theo4.6}
\end{thm}
%

\vspace{-1ex}\noindent\emph{Proof.}
Since the action of a linear operator is fully
determined by its action on elements of a basis,
it suffices to show that the equality holds on
elements of any basis.  Letting the operator on
the left-hand side act on $\ket{u_p}$, we have
\begin{equation*}
\left( \sum_{l=1}^n \ket{u_l}\bra{u_l} \right) \ket{u_p}
= \sum_{l=1}^n \ket{u_l} (\braket{u_l}{u_p})
= \sum_{l=1}^n \ket{u_l} \delta_{lp} = \ket{u_p} .
\tag*{\qed}
\end{equation*}

\smallskip
\begin{defn}
A bicomplex linear operator $H$ is called \emph{self-adjoint}
if $H^*=H$.
\label{D4a}
\end{defn}

\begin{lem}
Let $H:M\to M$ be a self-adjoint operator.
Then $P_k(H):V\to V$ $(k = 1, 2)$ is a self-adjoint
operator on $V$.
\label{Lemma4.2}
\end{lem}
\begin{proof}
By~\eqref{4.15}, $P_k(H)^*=P_k(H^*)=P_k(H)$.
\end{proof}

\begin{thm}
Two eigenkets of a bicomplex self-adjoint operator
are orthogonal if the difference of the two eigenvalues
is not in~$\nc$.
\label{Theo4.8}
\end{thm}
\begin{proof}
Let $H:M\to M$ be a self-adjoint operator and let
$\ket{\phi}$ and $\ket{\phi'}$ be two eigenkets of $H$
associated with eigenvalues $\lambda$ and $\lambda'$,
respectively. Then
\begin{align*}
0&=\scalarmath{\ket{\phi}}{H\ket{\phi'}}
-\scalarmath{\ket{\phi'}}{H\ket{\phi}}^{\dag_3}
=\lambda'\scalarmath{\ket{\phi}}{\ket{\phi'}}
-\[ \lambda\scalarmath{\ket{\phi'}}{\ket{\phi}} \]^{\dag_3}\\
&=\lambda'\scalarmath{\ket{\phi}}{\ket{\phi'}}
-\lambda^{\dag_3}\scalarmath{\ket{\phi'}}{\ket{\phi}}^{\dag_3}
=(\lambda'-\lambda^{\dag_3})\scalarmath{\ket{\phi}}{\ket{\phi'}}.
\end{align*}
Because $H$ is self-adjoint we know, from Theorem~14
of~\cite{Rochon3}, that $\lambda\in\D$.  Hence $\lambda^{\dag_3}=\lambda$
and if $\lambda'-\lambda\notin\nc$, then
$\scalarmath{\ket{\phi}}{\ket{\phi'}}=0$.
\end{proof}

With the structure we have now built, we
can prove the spectral decomposition theorem
for finite-dimensional bicomplex Hilbert spaces.
\begin{thm}
Let $M$ be a finite-dimensional free $\mathbb{T}$-module
and let $H:M\to M$ be a bicomplex self-adjoint operator.
It is always possible to find a set $\{ \ket{\phi_l} \}$
of eigenkets of $H$ that make up an orthonormal basis of $M$.
Moreover, $H$ can be expressed as
\begin{equation}
H=\sum_{l=1}^n{\lambda_l \ket{\phi_l}\bra{\phi_l}} ,
\label{4.16}
\end{equation}
where $\lambda_l$ is the eigenvalue of $H$ associated with
the eigenket $\ket{\phi_l}$.
\label{Theo Spectral}
\end{thm}
%
\noindent\emph{Proof.}
We first remark that the classical spectral decomposition theorem
holds for the self-adjoint operator $P_{k}(H)=H_\h{k}$, restricted to~$V$
($k=1,2$).  So let $\{ \ket{\phi_l}_\hh \}$ and $\{ \ket{\phi_l}_\hhh \}$
be orthonormal sets of eigenvectors of $H_\hh$ and $H_\hhh$,
respectively.  They make up orthonormal bases of~$V$ with respect
to the scalar products $\scalarmath{\cdot}{\cdot}_{\widehat{1}}$
and $\scalarmath{\cdot}{\cdot}_{\widehat{2}}$. Letting
$\ket{\phi_l}:= \ee \ket{\phi_l}_\hh + \eee \ket{\phi_l}_\hhh$,
we can see that $\{ \ket{\phi_l} \}$ is an orthonormal basis of $M$.
Let $\lambda_l$ be the eigenvalue of $H$ associated with
$\ket{\phi_l}$, so that $H\ket{\phi_l}=\lambda_l\ket{\phi_l}$.
To show that~\eqref{4.16} holds, it is enough to show that
the right-hand side of~\eqref{4.16} acts on basis kets
like $H$.  But
\begin{equation*}
\[ \sum_{l=1}^n \lambda_l \ket{\phi_l} \bra{\phi_l} \] \ket{\phi_p}
= \sum_{l=1}^n \lambda_l \ket{\phi_l}
(\braket{\phi_l}{\phi_p})
= \sum_{l=1}^n \lambda_l \delta_{lp} \ket{\phi_l}
= \lambda_p \ket{\phi_p} .
\tag*{\qed}
\end{equation*}

\section{Applications}

As an application of the results obtained in the
previous sections, we will develop the bicomplex version
of the quantum-mechanical evolution operator.
To do this, we first need to define bicomplex
unitary operators as well as functions of a bicomplex operator.

\subsection{Unitary Operators}

\begin{defn}
A bicomplex linear operator $U$ is called \emph{unitary}
if $U^* U = I$.
\label{D5a}
\end{defn}

From Definition~\ref{D5a} one easily sees that the action
of a bicomplex unitary operator preserves scalar products.
Indeed let $\ket{\psi},\ket{\phi}\in M$ and let $U$ be unitary.
Then
\begin{equation}
\scalarmath{U\ket{\psi}}{U\ket{\phi}}
= \scalarmath{U^*U \ket{\psi}}{\ket{\phi}}
= \scalarmath{I \ket{\psi}}{\ket{\phi}}
= \scalarmath{\ket{\psi}}{\ket{\phi}} .
\end{equation}

\begin{lem}
Let $U:M\to M$ be a unitary operator. Then $P_k(U):V\to V$
$(k = 1, 2)$ is a unitary operator on~$V$.
\label{Lemma4.3}
\end{lem}
%
\noindent\emph{Proof.}
From~\eqref{4.15} and Theorem~\ref{Theo3.8} we can write
\begin{equation*}
P_k(U)^* P_k(U) = P_k(U^*) P_k(U) = P_k(U^* U)
= P_k(I) = I .
\tag*{\qed}
\end{equation*}

We note that a bicomplex unitary operator cannot be
in the null cone.  For if it were, its determinant would
also be in the null cone and the operator would not have
an inverse.

\begin{thm}
Any eigenvalue $\lambda$ of a bicomplex unitary operator
satisfies $\lambda^{\dag_3}\lambda = 1$.
\label{Theo4.10}
\end{thm}
\begin{proof}
Let $\ket{\phi}\in M$ be an eigenket of a unitary
operator $U$, associated with the eigenvalue $\lambda$,
so that $U\ket{\phi}=\lambda\ket{\phi}$. Since $U$ preserves
scalar products, we can write
\begin{equation*}
\scalarmath{\ket{\phi}}{\ket{\phi}}
= \scalarmath{U\ket{\phi}}{U\ket{\phi}}
= \scalarmath{\lambda\ket{\phi}}{\lambda\ket{\phi}}
= \lambda^{\dag_3}\lambda \scalarmath{\ket{\phi}}{\ket{\phi}} .
\end{equation*}
Since an eigenket is not in the null cone,
$\lambda^{\dag_3}\lambda=1$ or, equivalently,
$\lambda^{\dag_3}=\lambda^{-1}$.
\end{proof}

\begin{cor}
Let $U$ be a unitary operator and let $\ket{\phi}\in M$
be an eigenket of $U$ associated with the eigenvalue $\lambda$.
Then $U^*\ket{\phi}=\lambda^{\dag_3}\ket{\phi}$.
\label{Corollary4.2}
\end{cor}
\begin{proof}
Because $U$ is unitary, one can write
\begin{equation*}
\lambda^{\dag_3}\ket{\phi} = \lambda^{\dag_3}I\ket{\phi}
= \lambda^{\dag_3}U^*U\ket{\phi}
= \lambda^{\dag_3}\lambda U^*\ket{\phi} .
\end{equation*}
The result follows from Theorem~\ref{Theo4.10}.
\end{proof}

\begin{thm}
Two eigenkets of a bicomplex unitary operator are
orthogonal if the difference of the eigenvalues is not in $\nc$.
\label{Theo4.11}
\end{thm}
\begin{proof}
Let $U:M\to M$ be a unitary operator and $\ket{\phi},\ket{\phi '}$
be two eigenkets of $U$ associated with eigenvalues $\lambda$
and $\lambda '$, respectively. Corollary~\ref{Corollary4.2}
then implies
\begin{align*}
0 &=\scalarmath{\ket{\phi}}{U\ket{\phi'}}
- \scalarmath{\ket{\phi'}}{U^*\ket{\phi}}^{\dag_3}
= \lambda'\scalarmath{\ket{\phi}}{\ket{\phi'}}
- \[ \lambda^{\dag_3}\scalarmath{\ket{\phi'}}{\ket{\phi}} \]^{\dag_3}\\
&= \lambda'\scalarmath{\ket{\phi}}{\ket{\phi'}}
- \lambda\scalarmath{\ket{\phi}}{\ket{\phi'}}
= (\lambda'-\lambda)\scalarmath{\ket{\phi}}{\ket{\phi'}}.
\end{align*}
If $\lambda'-\lambda\notin\nc$, then $\scalarmath{\ket{\phi}}{\ket{\phi'}}=0$.
\end{proof}

\subsection{Functions of an Operator}

Let $M$ be a finite-dimensional free $\T$-module and let
$A$ be a linear operator acting on $M$.  Let $A^0 := I$ and
let $\{ c_n ~|~ n=0, 1, \ldots \}$ be an infinite sequence of
bicomplex numbers.  Formally we can write the infinite sum
\begin{equation}
\sum_{n=0}^\infty c_n A^n .
\label{5.1}
\end{equation}
When this series converges to an operator acting on $M$,
we call this operator $f(A)$.

The operator $A$ and the coefficients $c_n$ can be
written in the idempotent basis as
\begin{equation}
A = \ee A_{\hh} + \eee A_{\hhh} , \qquad
c_n = \ee c_{n\hh} + \eee c_{n\hhh} .
\label{5.1a}
\end{equation}
Substituting \eqref{5.1a} into~\eqref{5.1}, we get
\begin{align}
f(A) &= \sum_{n=0}^\infty c_n A^n
= \ee \sum_{n=0}^\infty c_{n\hh} A_\hh^n
+ \eee \sum_{n=0}^\infty c_{n\hhh} A_\hhh^n \notag \\
&= \ee f_1 (A_\hh) + \eee f_2 (A_\hhh) .
\label{5.1b}
\end{align}
One can see that the $f$ series converges
if and only if the two series $f_1$ and $f_2$
converge.  These two are power series of
operators acting in a finite-dimensional
complex vector space.

A very important bicomplex function of an operator is
of course the \emph{exponential}, defined in the usual way as
\begin{equation}
\exp{A}=I+\sum_{n=1}^\infty{\frac{1}{n!}A^n} .
\label{5.2}
\end{equation}
Clearly,
\begin{equation}
\exp{A} = \ee \, \exp{A_\hh} + \eee \, \exp{A_\hhh} .
\label{5.2a}
\end{equation}

We now prove two important theorems on exponentials
of operators.

\begin{thm}
If $t$ is a real parameter,
$\frac{d}{dt} \exp{tA} = A \exp{tA}$.
\label{Theo5.1a}
\end{thm}
%
\noindent\emph{Proof.}
\begin{align*}
\frac{d}{dt} \exp{tA}
&= \frac{d}{dt} \left[ \ee \, \exp{t A_\hh}
+ \eee \, \exp{t A_\hhh} \right] \\
&= \ee A_\hh \, \exp{t A_\hh} + \eee A_\hhh \, \exp{t A_\hhh}
= A \, \exp{tA} .
\tag*{\qed}
\end{align*}

\medskip
\begin{thm}
If $H$ is self-adjoint, $\exp{\ii H}$ is unitary.
\label{Theo5.2a}
\end{thm}
%
\noindent\emph{Proof.}
Since $H_\hh$ and $H_\hhh$ are self-adjoint in the usual
(complex) sense, we have
\begin{align*}
&\left[ \exp{\ii H} \right]^* \exp{\ii H} \\
&= \left[ \ee \, \exp{\ii H_\hh} + \eee \, \exp{\ii H_\hhh} \right]^*
\left[ \ee \, \exp{\ii H_\hh} + \eee \, \exp{\ii H_\hhh} \right] \\
&= \left[ \ee \, \exp{-\ii H_\hh} + \eee \, \exp{-\ii H_\hhh} \right]
\left[ \ee \, \exp{\ii H_\hh} + \eee \, \exp{\ii H_\hhh} \right] \\
&= \ee \, \exp{-\ii H_\hh} \exp{\ii H_\hh}
+ \eee \, \exp{-\ii H_\hhh} \exp{\ii H_\hhh} \\
&= \ee I + \eee I = I .
\tag*{\qed}
\end{align*}

\subsection{Evolution Operator}

A generalization of the Schr\"odinger equation
to bicomplex numbers was proposed in~\cite{Rochon2}.
It can be adapted to finite-dimensional modules as
\begin{equation}
\ii\hbar\frac{d}{dt}\ket{\psi(t)}
= H\ket{\psi(t)} , \label{5.3}
\end{equation}
where $H$ is a self-adjoint bicomplex operator
(called the Hamiltonian).
Note that there is no gain in generality
if one adds an arbitrary invertible bicomplex constant
$\xi$ on the left-hand side, i.e.
\begin{equation}
\ii \xi \hbar\frac{d}{dt}\ket{\psi(t)}
= H\ket{\psi(t)} . \label{5.3a}
\end{equation}
Indeed one can then write
\begin{equation}
\ii \hbar\frac{d}{dt}\ket{\psi(t)}
= H' \ket{\psi(t)} , \label{5.3b}
\end{equation}
with $H' = \xi^{-1} H$.  For $H'$ to be self-adjoint
one must have $\xi^{\dag_3} = \xi$, so that
$\xi = \ee \xi_\hh + \eee \xi_\hhh$, with $\xi_\hh$
and $\xi_\hhh$ real.  In this case~\eqref{5.3a}
amounts to~\eqref{5.3} with a redefinition of the
Hamiltonian.

From Theorems~\ref{Theo5.1a} and~\ref{Theo5.2a} we
immediately obtain

\begin{thm}
If $H$ doesn't depend on time, solutions of~\eqref{5.3}
are given by $\ket{\psi (t)} = U(t, t_0) \ket{\psi (t_0)}$,
where $\ket{\psi (t_0)}$ is any ket and
\begin{equation*}
U(t, t_0) = \exp{- \frac{\ii}{\hbar} (t-t_0) H} .
\end{equation*}
\label{Theo5.3a}
\end{thm}
The operator $U(t, t_0)$ is unitary and is a generalization of the
\emph{evolution operator} of standard quantum
mechanics~\cite{Marchildon}.

\section{Conclusion}

We have derived a number of new results on
finite-dimensional bicomplex matrices, modules,
operators and Hilbert spaces, including the generalization
of the spectral decomposition theorem.  All these
concepts are deeply connected with the formalism
of quantum mechanics.  We believe that many if not
all of them can be extended to infinite-dimensional
Hilbert spaces.


\end{document}